\documentclass[12pt,reqno]{amsart}

\usepackage[dvips]{graphicx}
\usepackage[arrow,matrix,curve]{xy} 

\usepackage{amssymb, latexsym, amsmath, amscd, array, 
}

\newtheorem{theorem}{Theorem}[section]
\newtheorem{lemma}[theorem]{Lemma}
\newtheorem{proposition}[theorem]{Proposition}
\newtheorem{corollary}[theorem]{Corollary}

\theoremstyle{definition}
\newtheorem{definition}[theorem]{Definition}

\newtheorem{remark}[theorem]{Remark}
\newtheorem{question}[theorem]{Question}

\numberwithin{equation}{section}
\numberwithin{figure}{section} 
\numberwithin{table}{section}

\DeclareMathOperator{\PD}{{\rm PD}}

\DeclareMathOperator{\sys}{{\rm sys}}

\DeclareMathOperator{\trace}{{\rm trace}}

\DeclareMathOperator{\area}{{\rm area}}
\DeclareMathOperator{\vol}{{\rm vol}} 

\DeclareMathOperator{\PPP}{{\PP_{{\rm BA}}^{\phantom{I}}}}
\DeclareMathOperator{\av}{{av}} 
\DeclareMathOperator{\LA}{{LA}}
\DeclareMathOperator{\LB}{{LB}} 
\DeclareMathOperator{\TTT}{{T}}

\newcommand{\fla}{{f_{\LA}}} 
\newcommand{\hav}{{h_{\av}}}
\newcommand{\phiav}{{\phi_{\av}}}

\newcommand{\barf}{{\bar f\,}}

\newcommand\ie {{\it i.e.\ }}

\newcommand\cf {\hbox{\it cf.\ }}

\newcommand\gmetric{{\mathcal G}}

\newcommand{\length}{{\rm length}}

\newcommand \cqfd{\unskip\kern 6pt\penalty 500
\raise -2pt\hbox{\vrule\vbox to10pt{\hrule width 4pt
\vfill\hrule}\vrule}\par}                 

\def\adots{\mathinner{\mkern2mu\raise1pt\hbox{.}
\mkern3mu\raise4pt\hbox{.}\mkern1mu\raise7pt\hbox{.}}}
\def\hfl#1{\frac{\buildrel{#1}}{{\hbox to 12mm{\rightarrowfill}}}}
  
\def \\R^n \times \R^n
\rightarrow \R{\mathop{\R^n \times \R^n
\rightarrow \R}}

\newcommand\rk{\operatorname{{rank}}}

\newcommand\C{{\mathbb {C}}}

\newcommand\N {{\mathbb N}}

\newcommand\PP {{\mathbb {P}}}

\newcommand\var {{\rm Var}}

 \newcommand\R {{\mathbb R}}
\newcommand\RR {{\mathbb R}} 
\newcommand\T {{\mathbb T}} \newcommand\Z {{\mathbb Z}}

\long\def\forget#1\forgotten{} %

\long\def\forgett#1\forgottent{} %

\def\circ{\mathchoice%
 {\mathrel{\raise 1pt\hbox{$\scriptstyle\mathchar"020E$}}}
 {\mathrel{\raise 1pt\hbox{$\scriptstyle\mathchar"020E$}}}
 {\mathrel{\raise 1pt\hbox{$\scriptscriptstyle\mathchar"020E$}}}
 {}
}

\newcommand{\nc}{\newcommand} \nc{\on}{\operatorname}

\nc{\df}{\on{\it df}}

\nc{\conf}{\on{conf}}

\nc{\spt}{\on{spt}}
\nc{\norm}[1]{\| #1 \|}
\nc{\parallelleer}{\norm{\ }} 
\nc{\parallelh}{\norm h} 
\nc{\parallelk}{\norm k} 
\nc{\parallelx}{\norm x} 
\nc{\parallelhrr}{\norm {h_\RR}} 
\nc{\parallelom}{\norm \omega} 
\nc{\parallelomij}{\norm {\omega_{i_j}}} 
\nc{\parallelomx}{\norm {\omega_{x}}} 
\nc{\parallelpi}{\norm \pi} 
\nc{\parallelalf}{\norm \alpha} 
\nc{\parallelalfs}{\norm {\alpha_s}} 
\nc{\parallelalfi}{\norm {\alpha_i}} 
\nc{\parallelalfij}{\norm {\alpha_{i_j}}} 
\nc{\parallelbeta}{\norm \beta} 
\nc{\parallelbetat}{\norm {\beta_t}} 
\nc{\parallelhcapalf}{\norm {h \cap \alpha}} 
\nc{\parallelPDralf}{\norm {\PD_\RR(\alpha)}} 
\nc{\strichleer}{| \  |}
\nc{\NN}{\mathbb N}

\nc{\rr}{\mbox{$\scriptstyle\mathbb R$}}

\nc{\dF}{{\it dF}} 

\nc{\DF}{{\it DF}} 

\nc{\ds}{{\it ds}} 

\nc{\dvol}{{\it dvol}}

\nc{\grad}{{\rm grad}} 
\nc{\strichw}{\|\omega\|} 
\nc{\strichwx}{|\omega_x|}
\nc{\Hess}{{\rm Hess}}

\begin{document}

\title[Loewner's torus inequality with isosystolic defect] {Loewner's
torus inequality with isosystolic defect}

\author{Charles Horowitz}

\author{Karin Usadi Katz}

\author[M.~Katz]{Mikhail G. Katz$^{*}$}

\address{Department of Mathematics, Bar Ilan University, Ramat Gan
52900 Israel} \email{\{horowitz,katzmik\}@macs.biu.ac.il}

\thanks{$^{*}$Supported by the Israel Science Foundation (grants
no.~84/03 and 1294/06) and the BSF (grant 2006393)}



\date{\today}

\begin{abstract}
We show that Bonnesen's isoperimetric defect has a systolic analog for
Loewner's torus inequality.  The isosystolic defect is expressed in
terms of the probabilistic variance of the conformal factor of the
metric~$\gmetric$ with respect to the flat metric of unit area in the
conformal class of~$\gmetric$.
\end{abstract}

\maketitle \tableofcontents

\section{Bonnesen defect and isosystolic defect}

\begin{figure}
\includegraphics[height=2.5in]{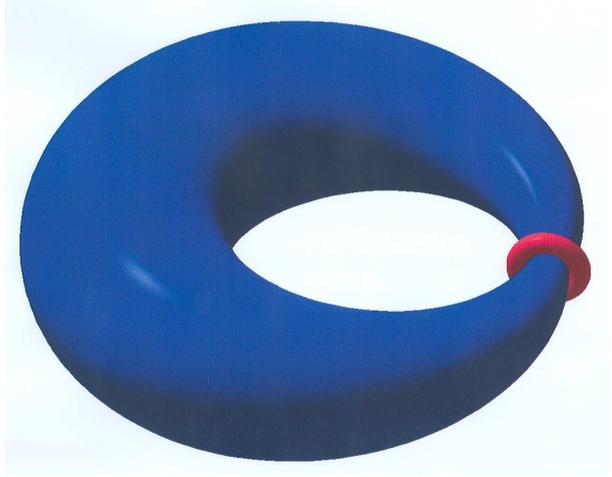}
\caption{Wikiartist's conception of a shortest loop on a torus}
\label{micro}
\end{figure}

The {\em systole\/} of a compact metric space~$X$ is a metric
invariant of~$X$, defined to be the least length of a noncontractible
loop in~$X$.  We will denote it~$\sys=\sys(X)$, \cf
M.~Gromov~\cite{Gr1, Gr2, Gr3, Gr4}.  When~$X$ is a graph, the
invariant is usually referred to as the {\em girth\/}, ever since
W.~Tutte's article~\cite{Tu}.  Possibly inspired by the latter,
C.~Loewner started thinking about systolic questions on surfaces in
the late forties, resulting in a~'50 thesis by his student P.M. Pu,
published as \cite{Pu}.

Loewner himself did not publish his torus inequality~\eqref{11L},
apparently leaving it to Pu to pursue this line of research.
Meanwhile, the latter was recalled to the mainland after the
communists ousted Chiang Kai-shek in '49.  Pu was henceforth confined
to research in fuzzy topology in the service of the people.  Our guess
is that Pu may have otherwise obtained a geometric inequality with
isosystolic defect, already half a century ago, placing it among the
classics of the global geometry of surfaces.

Similarly to the isoperimetric inequality, Loewner's torus inequality
relates the total area, to a suitable~$1$-dimensional invariant,
namely the systole, \ie least length of a noncontractible loop on the
torus~$(\T^2, \gmetric)$:
\begin{equation}
\label{11L} 
\area(\gmetric) - \tfrac{\sqrt{3}}{2} \sys(\gmetric)^2 \geq 0,
\end{equation}
\cf \eqref{51} and \cite{Pu, SGT}.

The classical Bonnesen inequality \cite{Bo} is the strengthened
isoperimetric inequality
\begin{equation}
\label{11b}
L^2 - 4\pi A \geq \pi^2(R-r)^2,
\end{equation}
see \cite[p.~3]{B-Z88}.  Here~$A$ is the area of the region bounded by
a closed Jordan curve of length (perimeter)~$L$ in the plane,~$R$ is
the circumradius of the bounded region, and~$r$ is its inradius.  The
error term~$\pi^2(R-r)^2$ on the right hand side of~\eqref{11b} is
traditionally referred to as the {\em isoperimetric defect}.

In the present text, we will strengthen Loewner's torus inequality by
introducing a ``defect" term \`a la Bonnesen.  There is no defect term
in either \cite{Pu} or \cite{SGT}.  The approach that has been used in
the literature is via an integral identity expressing area in terms of
energies of loops.  Somehow researchers in the field seem to have
overlooked the fact that the computational formula for the variance
yields an improvement, namely the defect term.  There is thus a
significant change of focus, from the integral geometric identity, to
the application of the computational formula, elementary though it may
be.

If we use conformal representation to express the metric~$\gmetric$ on
the torus as
\[
        f^2 (dx^2+dy^2)
\]
with respect to a unit area flat metric~$dx^2+dy^2$ on the torus
viewed as a quotient of the~$(x,y)$ plane by a lattice (see
\eqref{uni}), then the defect term in question is simply the variance
of the conformal factor~$f$ above.  Then the inequality with the
defect term can be written as follows:
\begin{equation}
\label{13}
\area(\gmetric) - \tfrac{\sqrt{3}}{2} \sys(\gmetric)^2 \geq \var(f).
\end{equation}
Here the error term, or {\em isosystolic defect}, is given by the
variance
\begin{equation}
\label{14b}
\var(f)=\int_{\T^2} (f-m)^2
\end{equation}
of the conformal factor~$f$ of the metric~$\gmetric=f^2(dx^2+dy^2)$ on
the torus, relative to the unit area flat
metric~$\gmetric_0=dx^2+dy^2$ in the same conformal class.  Here
\begin{equation}
\label{15b}
m=\int_{\T^2} f
\end{equation}
is the mean of~$f$.  More concretely, if~$(\T^2,\gmetric_0) = \R^2/L$
where~$L$ is a lattice of unit coarea, and~$D$ is a fundamental domain
for the action of~$L$ on~$\R^2$ by translations, then the integral
\eqref{15b} can be written as
\[
m=\int_D f(x,y) dxdy
\]
where~$dxdy$ is the standard measure of~$\R^2$.  Every flat torus is
isometric to a quotient~$\T^2 = \R^2/L$ where~$L$ is a lattice, \cf
\cite[Theorem 38.2]{Loe}.  Recall that the uniformisation theorem in
the genus~$1$ case can be formulated as follows.

\begin{theorem}[Uniformisation theorem]
\label{11u}
For every metric~$\gmetric$ on the~$2$-torus~$\T^2$, there exists a
lattice~$L\subset \R^2$ and a positive~$L$-periodic function~$f(x,y)$
on~$\R^2$ such that the torus~$(\T^2,\gmetric)$ is isometric to
\begin{equation}
\label{uni}
\left( \R^2/L, f^2 ds^2 \right),
\end{equation}
where~$ds^2=dx^2+dy^2$ is the standard flat metric of~$\R^2$.
\end{theorem}

When the flat metric is that of the unit square torus, Loewner's
inequality can be strengthened to the inequality
\[
\area(\gmetric) - \sys(\gmetric)^2 \geq \var(f),
\]
\cf \eqref{52z}.  In this case, if the conformal factor depends only
on one variable (as, for example, in the case of surfaces of
revolution), one can strengthen the inequality further by providing a
second defect term as follows:
\begin{equation}
\area(\gmetric) - \sys(\gmetric)^2 \geq \var(f) + \frac{1}{4} \left|
f_0 \right|_1^2,
\end{equation}
where~$f_0=f-E(f)$, while~$E(f)$ is the expected value of~$f$, and
$\left|\; \right|_1$ is the~$L^1$-norm.  See also
inequality~\eqref{93}.  More generally, we obtain the following
theorem.

We first define a ``biaxial'' projection~$\PPP(f)$ as follows.  Given
a doubly periodic function~$f(x,y)$, \ie a function defined on
$\R^2/\Z^2$, we decompose~$f$ by setting
\[
f(x,y)= E(f) + g_f(x) + h_f(y) + k_f(x,y),
\]
where the single-variable functions~$g_f$ and~$h_f$ have zero means,
while~$k_f$ has zero mean along every vertical and horizontal unit
interval.  We have~$g_f(x)=\int_0^1 f(x,y)dy$, while~$h_f(y)=\int_0^1
f(x,y)dx$.  The projection~$\PPP(f)$ is then defined by setting
\[
\PPP(f)= g_f(x)+h_f(y).
\]
In terms of the double Fourier series of~$f$, the projection~$\PPP$
amounts to extracting the~$(m,n)$-terms such that~$mn=0$ (\ie the
terms located along the pair of coordinate axes),
but~$(m,n)\not=(0,0)$.

\begin{theorem}
\label{11c}
In the conformal class of the unit square torus, the metric~$f^2 ds^2$
defined by a general conformal factor~$f(x,y)>0$, satisfies the
following version of Loewner's torus inequality with a second systolic
defect term:
\begin{equation}
\label{16}
\area(\gmetric) - \sys(\gmetric)^2 \geq \var(f) + \frac{1}{16} \left|
\PPP(f) \right|_1^2.
\end{equation}
\end{theorem}

Theorem~\ref{11c} is proved in Section~\ref{ten}.  

Marcel Berger's monograph \cite[ pp.~325-353]{Be6} contains a detailed
exposition of the state of systolic affairs up to '03.  More recent
developments are covered in \cite{SGT}.  Recent publications in
systolic geometry include \cite{Be08, Bru, Bru2, Bru3, DKR, Ka4, RS,
Sa08, e7, AK, KK, KSh}.

\section{Variance, Hermite constant, successive minima}
\label{L333}

The proof of inequalities with isosystolic defect relies upon the
familiar computational formula for the variance of a random variable
in terms of expected values.  Keeping our differential geometric
application in mind, we will denote the random variable~$f$.  Namely,
we have the formula
\begin{equation}
\label{11}
E_\mu(f^2) - \left(E_\mu(f) \right)^2 = \var(f),
\end{equation}
where~$\mu$ is a probability measure.  Here the variance is
\[
\var(f)= E_\mu \left( (f-m)^2 \right),
\]
where~$m=E_\mu(f)$ is the expected value (\ie the mean).

Now consider a flat metric~$\gmetric_0$ of unit area on
the~$2$-torus~$\T^2$.  Denote the associated measure by~$\mu$.
Since~$\mu$ is a probability measure, we can apply formula~\eqref{11}
to it.  Consider a metric~$\gmetric = f^2 \gmetric_0$ conformal to the
flat one, with conformal factor~$f(x,y)>0$, and new measure~$f^2\mu$.
Then we have
\[
E_\mu (f^2) = \int_{\T^2} f^2 \mu = \area (\gmetric).
\]
Equation~\eqref{11} therefore becomes
\begin{equation}
\label{21}
\area(\gmetric) - \left( E_\mu(f) \right)^2 = \var(f).
\end{equation}

Next, we will relate the expected value~$E_\mu(f)$ to the systole of
the metric~$\gmetric$.  To proceed further, we need to deal with some
combinatorial preliminaries.  We will then relate \eqref{11} to
Loewner's torus inequality.

Let~$B$ be a finite-dimensional Banach space, \ie a vector space
together with a norm~$\|\;\|$.  Let~$L\subset (B,\|\;\|)$ be a lattice
of maximal rank, \ie satisfying~$\rk(L)=\dim(B)$.  We define the
notion of successive minima of~$L$ as follows.

\begin{definition}
\label{mindef}
For each~$k=1,2,\dots, \rk(L)$, define the {\em~$k$-th successive
minimum\/}\index{successive minima|textbf} of the lattice~$L$ by
\begin{equation}
\label{success}
\lambda_k(L,\|\;\|) = \inf\bigg\{\lambda\in\R \ \left| \,
\begin{array}{l}
\exists \text{ lin. indep. } v_1, \ldots, v_k \in L \\ \text{\ with\ }
\|v_i\|\leq \lambda \text{\ for\ all\ }i
\end{array}  
\right. \bigg\}.
\end{equation}
\end{definition}

Thus the first successive minimum,~$\lambda_1(L,\|\;\|)$ is the least
length of a nonzero vector in~$L$.

\begin{definition}
Let~$b\in \N$.  The {\em Hermite constant}~$\gamma_b$ is defined in
one of the following two equivalent ways:
\begin{enumerate}
\item
$\gamma_b$ is the {\em square\/} of the biggest first successive
minimum\index{successive minima}, \cf Defi\-nition~\ref{mindef}, among
all lattices of unit covolume;
\item
$\gamma_b$ is defined by the formula
\begin{equation}
\label{421}
\sqrt{\gamma_b} = \sup\left\{ \left. 
\frac{\lambda_1(L)}{\vol(\R^b/L) ^{{1}/{b}}} 
\right| L \subseteq (\R ^b, \|\;\|) \right\},
\end{equation}
where the supremum is extended over all lattices~$L$ in~$\R^b$ with a
Euclidean norm~$\|\;\|$.
\end{enumerate}
\end{definition}

A lattice realizing the supremum is called a {\em critical\/}
lattice.\index{critical lattice|textbf} A critical lattice may be
thought of as the one realizing the densest packing in~$\R^b$ when we
place balls of radius~$\frac{1}{2}\lambda_1(L)$ at the points of~$L$.

\section{Standard fundamental domain and Eisenstein integers}
\label{four}

\begin{definition}
The lattice of the {\em Eisenstein integers\/} is the lattice in~$\C$
spanned by the elements~$1$ and the sixth root of unity.
\end{definition}

To visualize the lattice, start with an equilateral triangle in~$\C$
with vertices~$0$,~$1$, and~$\tfrac{1}{2}+i\tfrac{\sqrt{3}}{2}$, and
construct a tiling of the plane by repeatedly reflecting in all sides.
The Eisenstein integers are by definition the set of vertices of the
resulting tiling.

The following result is well-known.  We reproduce a proof here since
it is an essential part of the proof of Loewner's torus inequality
with isosystolic defect.

\begin{lemma}
\label{217}
When~$b=2$, we have the following value for the Hermite
constant:~$\gamma_2=\frac{2}{\sqrt{3}}=1.1547\ldots$.  The
corresponding critical lattice is homothetic to the~$\Z$-span of the
cube roots of unity in~$\C$, \ie the Eisenstein integers.
\end{lemma}

\begin{proof}
Consider a lattice~$L\subset \C=\R^2$.  Clearly, multiplying~$L$ by
nonzero complex numbers does not change the value of the quotient
\[
\frac{\lambda_1(L)^2}{\area(\C/L)}.
\]
Choose a ``shortest'' vector~$z \in L$, \ie we have~$|z|= \lambda_1
(L)$.  By replacing~$L$ by the lattice~$z^{-1}L$, we may assume that
the complex number~$+1\in \C$ is a shortest element in the lattice.
We will denote the new lattice by the same letter~$L$, so that now
$\lambda_1(L)=1$.  Now complete the element~$+1\in L$ to a~$\Z$-basis
\begin{equation}
\label{31b}
\{\tau, +1\}
\end{equation}
for~$L$.  Thus~$|\tau|\geq \lambda_1(L)=1$.  Consider the real
part~$\Re(\tau)$.  Clearly, we can adjust the basis by adding a
suitable integer to~$\tau$, so as to satisfy the condition~$-\tfrac
{1}{2} \leq \Re(\tau)\leq \tfrac{1}{2}$.  Then the basis vector~$\tau$
lies in the closure of the standard fundamental domain
\begin{equation}
\label{fd}
D=\left\{ z\in\C \left|\; |z| > 1, \; | \Re(z) | < \tfrac{1}{2},\;
\Im(z) > 0 \right. \right\}
\end{equation}
for the action of the group~${\rm PSL}(2,\Z)$ in the upperhalf plane
of~$\C$.  The imaginary part satisfies~$\Im(\tau)\geq
\frac{\sqrt{3}}{2}$, with equality possible in the following two
cases:~$\tau=e^{i\frac{\pi}{3}}$ or~$\tau=e^{i\frac {2\pi}{3}}$.
Finally, we calculate the area of the parallelogram in~$\C$ spanned
by~$\tau$ and~$+1$, and write 
\[
\frac{\area(\C/L)}{\lambda_1(L)^2} = \Im (\tau) \geq
\frac{\sqrt{3}}{2}
\]
to conclude the proof.
\end{proof}

\section{Fundamental domain and Loewner's torus inequality}

We now return to the proof of Loewner's torus inequality for the
metric~$\gmetric = f^2 \gmetric_0$ using the computational formula for
the variance.  Let us analyze the expected value term~$E_\mu(f) =
\int_{\T^2} f \mu$ in \eqref{21}.

By the proof of Lemma~\ref{217}, the lattice of deck transformations
of the flat torus~$\gmetric_0$ admits a~$\Z$-basis similar to~$\{\tau,
1\} \subset \C$, where~$\tau$ belongs to the standard fundamental
domain \eqref{fd}.  In other words, the lattice is similar to
\[
\Z\tau + \Z1 \subset \C.
\]
Consider the imaginary part~$\Im(\tau)$ and set
\[
\sigma^2:=\Im(\tau)>0.
\]
From the geometry of the fundamental domain it follows that~$\sigma^2
\geq \tfrac{\sqrt{3}}{2}$, with equality if and only if~$\tau$ is the
primitive cube or sixth root of unity.  Since~$\gmetric_0$ is assumed
to be of unit area, the basis for its group of deck tranformations can
therefore be taken to be
\[
\{ \sigma^{-1}\tau, \sigma^{-1}\},
\]
where~$\Im ( {\sigma} ^{-1} {\tau} ) =\sigma$.  We will prove the
following generalisation of Loewner's bound.

\begin{theorem}
Every metric $\gmetric$ on the torus satisfies the inequality
\begin{equation}
\area(\gmetric) - \sigma^2 \sys(\gmetric)^2 \geq \var(f),
\end{equation}
where~$f$ is the conformal factor of the metric~$\gmetric$ with
respect to the unit area flat metric~$\gmetric_0$.  
\end{theorem}

\begin{proof}
With the normalisations described above, we see that the flat torus is
ruled by a pencil of horizontal closed geodesics,
denoted~$\gamma_y=\gamma_y(x)$, each of length~$\sigma^{-1}$, where
the ``width'' of the pencil equals~$\sigma$, i.e. the parameter~$y$
ranges through the interval~$[0,\sigma]$,
with~$\gamma_\sigma=\gamma_0$.

By Fubini's theorem, we obtain the following lower bound for the
expected value:
\[
\begin{aligned}
E_\mu(f) &= \int_0^\sigma \left( \int_{\gamma_y} f(x)dx \right) dy
\\&= \int_0^\sigma \length(\gamma_y)dy
\\&\geq \sigma \sys(\gmetric),
\end{aligned}
\]
Substituting into \eqref{21}, we obtain the inequality
\begin{equation}
\label{51b}
\area(\gmetric) - \sigma^2 \sys(\gmetric)^2 \geq \var(f),
\end{equation}
where~$f$ is the conformal factor of the metric~$\gmetric$ with
respect to the unit area flat metric~$\gmetric_0$.  
\end{proof}

Since~$\sigma^2 \geq \tfrac{\sqrt{3}}{2}$, we obtain in particular a
strengthening of Loewner's torus inequality, namely the following
inequality with isosystolic defect:
\begin{equation}
\label{51}
\area(\gmetric) - \tfrac{\sqrt{3}}{2} \sys(\gmetric)^2 \geq \var(f),
\end{equation}
as discussed in the introduction.

\begin{corollary}
A metric satisfying the boundary case of equality in Loewner's torus
inequality~\eqref{11L} is necessarily flat and homothetic to the
quotient of~$\R^2$ by the lattice of Eisenstein integers.
\end{corollary}

\begin{proof}
If a metric~$f^2 ds^2$ satisfies the boundary case of equality in
\eqref{11L}, then the variance of the conformal factor~$f$ must vanish
by \eqref{51}.  Hence~$f$ is a constant function.  The proof is
completed by applying Lemma~\ref{217}.
\end{proof}

Now suppose~$\tau$ is pure imaginary, \ie the lattice~$L$ is a
rectangular lattice of coarea~$1$.  Note that this property for a
coarea~$1$ lattice is equivalent to the equality~$\lambda_1(L)
\lambda_2(L)=1$.

\begin{corollary}
If~$\tau$ is pure imaginary, then the metric~$\gmetric=f^2 \gmetric_0$
satisfies the inequality
\begin{equation}
\label{52z}
\area(\gmetric) - \sys(\gmetric)^2 \geq \var(f).
\end{equation}
\end{corollary}

\begin{proof}
If~$\tau$ is pure imaginary then~$\sigma\geq 1$, and the inequality
follows from \eqref{51b}.
\end{proof}

In particular, every surface of revolution satisfies \eqref{52z},
since its lattice is rectangular, \cf Corollary~\ref{72}.

\section{First fundamental form
and surfaces of revolution}
\label{five}

This elementary section is concerned mainly with surfaces of
revolution and an explicit construction of isothermal coordinates on
such surfaces.  Recall that the first fundamental form of a regular
parametrized surface~$\underline x(u^1, u^2)$ in~$\R^3$ is the
bilinear form on the tangent plane defined by the restriction of the
ambient inner product~$\langle\;,\;\rangle$.  With respect to the
basis~$\{x_1, x_2\}$, where~$x_i= \frac{\partial x}{\partial u^i}$, it
is given by the two by two matrix~$(g_{ij})$, where~$g_{ij} = \langle
x_i,x_j \rangle$ are the metric coefficients.

In the special case of a surface of revolution, it is customary to use
the notation~$u^1 = \theta$ and~$u^2 = \varphi$.  The starting point
is a curve~$C$ in the~$xz$-plane, parametrized by a pair of
functions~$x = f(\varphi)$,~$z = g(\varphi)$.  We will assume
that~$f(\varphi)>0$.  The surface of revolution (around the~$z$-axis)
defined by~$C$ is parametrized as follows :~$\underline
x(\theta,\varphi) = (f(\varphi) \cos \theta, f(\varphi) \sin \theta,
g(\varphi))$.  The condition~$f(\varphi)>0$ ensures that the resulting
surface is an imbedded torus, provided the original curve~$C$ itself
is a Jordan curve.  The pair of functions~$(f,g)$ gives an arclength
parametrisation of the curve if~$\left(\frac{df}{d\varphi}\right)^2 +
\left(\frac{dg}{d\varphi}\right)^2 =1$.  For example,
setting~$f(\varphi) = \sin \varphi$ and~$g(\varphi) = \cos \varphi$,
we obtain a parametrisation of the sphere~$S^2$ in spherical
coordinates.  To calculate the first fundamental form of a surface of
revolution, note that~$x_1 = {\partial x \over \partial \theta} =(-f
\sin \theta, f \cos \theta, 0)$, while~$x_2 = {\partial x \over
\partial \varphi} = \left({df \over d\varphi} \cos \theta, {df \over d
\varphi} \sin \theta, {dg \over d\varphi} \right)$, so that we
have~$g_{11} = f^2 \sin^2 \theta + f^2 \cos^2 \theta = f^2$, while
$g_{22} = \left({df \over d\varphi} \right)^2 (\cos^2 \theta + \sin^2
\theta) + \left({dg \over d\varphi}\right)^2 = \left({df \over
d\varphi}\right)^2 + \left({dg \over d\varphi}\right)^2$ and~$g_{12} =
-f {df \over d\varphi} \sin \theta \cos \theta + f {df \over d\varphi}
\cos \theta \sin \theta = 0$.  Thus we obtain the first fundamental
form
\begin{equation}
\label{64}
(g_{ij}) =
\begin{pmatrix}
f^2 & 0 \cr 0 & \left({df \over d\varphi}\right)^2 + \left({dg \over
 d\varphi}\right)^2 \cr
\end{pmatrix}.
\end{equation}
We have the following obvious lemma.
\begin{lemma}
For a surface of revolution obtained from a unit speed parametrisation
$(f(\varphi),g(\varphi))$ of the generating curve, we obtain the
following matrix of the coefficients of the first fundamental form:
\[
(g_{ij}) =
\begin{pmatrix}
f^2 & 0 \cr 0 & 1 \cr
\end{pmatrix}.
\]
\end{lemma}

The following lemma expresses the metric of a surface of revolution in
isothermal coordinates.

\begin{lemma}
Suppose~$(f(\varphi),g(\varphi))$, where~$f(\varphi)> 0$, is an
arclength parametrisation of the generating curve of a surface of
revolution.  Then the change of variable
\[
\psi= \int \frac{ d\varphi}{f(\varphi)}
\]
produces a new parametrisation (in terms of variables~$\theta, \psi$),
with respect to which the first fundamental form is given by a scalar
matrix~$(g_{ij})=(f^2\delta_{ij})$.
\end{lemma}

In other words, we obtain an explicit conformal equivalence between
the metric on the surface of revolution and the standard flat metric
on the quotient of the~$(\theta, \psi)$ plane.  Such coordinates are
referred to as ``isothermal coordinates'' in the literature.  The
existence of such a parametrisation is of course predicted by the
uniformisation theorem (see Theorem~\ref{11u}) in the case of a
general surface.

\begin{proof}
Let~$\varphi=\varphi(\psi)$.  By chain rule,~$\frac{df}{d\psi} =
\frac{df}{d\varphi} \frac{d\varphi}{d\psi}$.  Now consider again the
first fundamental form \eqref{64}.  To impose the
condition~$g_{11}=g_{22}$, we need to solve the equation~$f^2=
\left(\frac{df}{d\psi}\right)^2 + \left( \frac{dg}{d\psi}\right)^2$,
or 
\[
f^2= \left( \left(\frac{df}{d\varphi}\right)^2 + \left( \frac{dg}
{d\varphi}\right)^2 \right) \left(\frac{d\varphi}{d\psi}\right)^2.
\]
In the case when the generating curve is parametrized by arclength, we
are therefore reduced to the equation~$f= \frac{d\varphi}{d\psi}$, or
$\psi= \int \frac{d\varphi}{f(\varphi)}$.  Replacing~$\varphi$
by~$\psi$, we obtain a parametrisation of the surface of revolution in
coordinates~$(\theta, \psi)$, such that the matrix of metric
coefficients is a scalar matrix.
\end{proof}

\begin{corollary}
\label{72}
Consider a torus of revolution in~$\R^3$ formed by rotating a Jordan
curve with unit speed parametisation~$(f(\varphi), g(\varphi))$ where
$\varphi\in [0,L]$, and~$L$ is the total length of the closed curve.
Then the torus is conformally equivalent to a flat torus defined by a
rectangular lattice
\[
a \Z \oplus b \Z,
\]
where~$a= 2\pi$ and~$b=\int_0^L \frac{d\varphi}{f(\varphi)}$.
\end{corollary}

\section{A second isosystolic defect term}

In the notation of Section~\ref{four}, assume for simplicity that
$\tau=i$, \ie the underlying flat metric is that of a unit square
torus~$\R^2/\Z^2$ where we think of~$\R^2$ as the~$(x,y)$ plane.  For
metrics in this conformal class, we will obtain an additional defect
term for Loewner's torus inequality.  First, we study a
metric~$\gmetric=f^2 ds^2$, defined by a conformal factor~$f(y)>0$,
where~$ds^2 = dx^2+dy^2$ is the standard flat metric and the conformal
factor only depends on one of the variables, as in the case of a
surface of revolution, see Section~\ref{five}.  Our estimate is based
on the following lemma.

\begin{lemma}
\label{91}
Let~$g$ be a continuous function with zero mean on the unit
interval~$[0,1]$.  Then we have the following bound in terms of
the~$L^1$ norm:
\[
\int_0^1 (g - \min{\!}_g) \geq \tfrac{1}{2} \left| g \right|_1.
\]
\end{lemma}

\begin{proof}
Let~$S^+\subset [0,1]$ be the set where the function~$g$ is positive,
so that~$\left|g\right|_1 = \int|g| = 2 \int_{S^+} g$.  Since
$\min{\!}_g \leq 0$, we obtain
\[
\int_0^1 (g - \min{\!}_g) \geq \int_{S^+} (g -\min{\!}_g) \geq
\int_{S^+} g = \frac{1}{2} \left| g \right|_1,
\]
completing the proof of the lemma.
\end{proof}

Consider the unit square torus~$(\R^2/\Z^2, ds^2)$,
where~$ds^2=dx^2+dy^2$, covered by the~$(x,y)$ plane.

\begin{theorem}
\label{92}
If the conformal factor~$f$ of the metric~$\gmetric = f^2 ds^2$
on~$\R^2/\Z^2$ only depends on one of the two variables,
then~$\gmetric$ satisfies the inequality
\begin{equation}
\label{61}
\area(\gmetric) - \var(f) \geq \left( \sys(\gmetric) + \tfrac{1}{2}
\left|f_0 \right|_1 \right)^2,
\end{equation}
where~$f_0=f-m$ and~$m$ is the expected value of~$f$.
\end{theorem}

To make inequality \eqref{61} resemble Loewner's torus inequality, we
can rewrite it as follows:
\[
\area(\gmetric) - \sys(\gmetric)^2 \geq \var(f) + \sys(\gmetric)
\left|f_0\right|_1 + \frac{1}{4} \left|f_0\right|_1^2,
\]
so that, in particular, we obtain a form of the inequality which does
not involve the systole in the right hand side:
\begin{equation}
\label{93} 
\area(\gmetric) - \sys(\gmetric)^2 \geq \var(f) + \frac{1}{4}
\left|f_0\right|_1^2.
\end{equation}

\begin{proof}[Proof of Theorem~\ref{92}]
To fix ideas, assume~$f$ only depends on~$y$.  Let~$y_0$ be the point
where the minimum~$\min{\!}_f$ of~$f=f(y)$ is attained.
The~$\gmetric$-length of the horizontal unit interval at height~$y_0$
equals
\begin{equation}
\label{62}
\int_0^1 f(x,y_0)dx= \int_0^1 \min{\!}_f dx =\min{\!}_f.
\end{equation}
Such an interval parametrizes a noncontractible loop on the torus, and
we obtain
\[
\sys(\gmetric) = \min{\!}_f.
\]
Applying Lemma~\ref{91} to~$f_0=f-E(f)$ where~$f$ is the conformal
factor, we obtain
\begin{equation}
\label{64e}
E(f) -\sys(\gmetric) = \int_0^1 (f - \min{\!}_f) = \int_0^1 (f_0 -
\min{\!}_{f_0}) \geq \frac{1}{2} \left| f_0 \right|_1,
\end{equation}
and the theorem follows from \eqref{21}.
\end{proof}

\section{Biaxial projection and second defect}
\label{ten}

Now consider an arbitrary conformal factor~$f>0$ on~$\R^2/\Z^2$.  We
decompose~$f$ into a sum
\[
f(x,y)= E(f) + g_f(x) + h_f(y) + k_f(x,y),
\]
where functions~$g_f$ and~$h_f$ have zero means, and~$k_f$ has zero
mean along every vertical and horizontal unit interval.  The
``biaxial'' projection~$\PPP(f)$ is defined by setting
\begin{equation}
\label{101}
\PPP(f)= g_f(x)+h_f(y).
\end{equation}
In terms of the double Fourier series of~$f$, the projection~$\PPP$
amounts to extracting the~$(m,n)$-terms such that~$mn=0$ (\ie the pair
of axes), but~$(m,n)\not=(0,0)$.

\begin{theorem}
\label{11cc}
In the conformal class of the unit square torus, the metric~$f^2 ds^2$
defined by a conformal factor~$f(x,y)>0$, satisfies the following
version of Loewner's torus inequality with a second defect term:
\begin{equation}
\label{15bac}
\area(\gmetric) - \sys(\gmetric)^2 \geq \var(f) + \frac{1}{16} \left|
\PPP(f) \right|_1^2.
\end{equation}
If~$f$ only depends on one variable then the
coefficient~$\frac{1}{16}$ in \eqref{15bac} can be replaced
by~$\frac{1}{4}$.
\end{theorem}

\begin{proof}
Applying the triangle inequality to \eqref{101}, we obtain
\[
\left| \PPP(f) \right|_1 \leq \left| g_f(x) \right|_1 + \left| h_f(y)
\right|_1.
\]
Due to the symmetry of the two coordinates, we can assume without loss
of generality that
\begin{equation}
\label{102}
\left| h_f(y) \right|_1 \geq \frac{1}{2} \left| \PPP(f) \right|_1.
\end{equation}
We define a function~$\barf$ by setting
\[
\barf(y)=E(f) + h_f(y)= \int_0^1 f(x,y)dx.
\]
We have~$\barf >0$ since it is an average of a positive function.
Clearly, we have~$\barf_0 = h_f$.  By Lemma~\ref{91} applied
to~$\barf_0$, we obtain
\[
\int \left( \barf - \min{}_{\barf} \right) \geq \frac{1}{2} \left|
\barf_0 \right|_1 \geq \frac{1}{4} \PPP(f)
\]
in view of~\eqref{102}.  We now compare the two metrics~$\barf^2 ds^2$
and~$f^2ds^2$.  Let~$y_0$ be the point where the function~$\barf$
attains its minimum.  Then
\begin{equation}
\label{74}
\sys(\barf^2 ds^2) = \min{}_{\barf} = \barf(y_0) = \int_0^1 f(x,y_0)dx
\geq \sys(f^2ds^2).
\end{equation}
Meanwhile,
\begin{equation}
\label{75}
E(f)=E(\barf) \geq \sys( \barf^2 ds^2) + \frac{1}{2} \left| \barf_0
\right|_1
\end{equation}
by \eqref{64e} applied to the averaged metric~$\barf^2ds^2$.  Thus,
\[
\begin{aligned}
\area(f^2ds^2) -\var(f) &= E(f)^2 \cr& \geq \left( \sys(\barf^2ds^2) +
\frac{1}{4}\PPP(f) \right)^2 \cr&\geq \left( \sys(f^2ds^2) +
\frac{1}{4}\PPP(f) \right)^2
\end{aligned}
\]
by combining \eqref{74} and \eqref{75}.  
\end{proof}

\section{Acknowledgments}

We are grateful to M. Agranovsky, A. Rasin, and A. Reznikov for
helpful discussions.

\vfill\eject

\end{document}